\documentclass[12pt]{amsart}

\usepackage{DKstyle}

\newcommand{\vertiii}[1]{{\left\vert\kern-0.25ex\left\vert\kern-0.25ex\left\vert #1
    \right\vert\kern-0.25ex\right\vert\kern-0.25ex\right\vert}}
\theoremstyle{plain}

\usepackage{caption}
\usepackage[labelfont=rm]{subcaption}

\usepackage[pagebackref=true, colorlinks]{hyperref}

\hypersetup{pdffitwindow=true,linkcolor=blue,citecolor=blue,urlcolor=blue,menucolor=blue}

\usepackage{comment}

\begin{document}
\title{Approximation of points in the plane by generic lattice orbits}
\author{Dubi Kelmer}

\thanks{Dubi Kelmer is partially supported by NSF grant DMS-1401747.}
\email{kelmer@bc.edu}
\address{Boston College, Boston, MA}


\date{\today}%
\dedicatory{}%
\commby{}%

\begin{abstract}
We give upper and lower bounds for Diophantine exponents measuring how well a point in the plane can be approximated by points in the orbit of a lattice $\G<\SL_2(\R)$ acting linearly on $\R^2$. Our method gives bounds that are uniform for almost all orbits.
\end{abstract}

 \maketitle
 \section{Introduction}
 Let $\G\subseteq \SL_2(\R)$ be a lattice and for each $T>0$ let $\G_T=\{\g\in \G: \|\g\|\leq T\}$ with $\|\g\|=\tr(\g^t\g)$ the Hilbert-Schmidt norm.  For any $u\in \R^2\setminus \{0\}$ and $T>0$ consider the finite orbit 
 $ \G_T u$ where $\G$ acts linearly on $\R^2$.
The limiting distribution of these orbits as $T\to\infty$ was extensively studied in \cite{Ledrappier99,Nogueira02,GorodnikWeiss07} and shown to be equidistributed with respect to a suitable measure (depending on $u$). In particular,  for a generic point $u\in \R^2$, the  orbit $\G u$ is dense in $\R^2$, and hence any point $v \in \R^2$ can be approximated by orbit points (when $\G$ is co-compact all orbits are dense). 
 
To measure how well a point $v\in \R^2$ can be approximated by orbit points in $\G u$, in analogy to similar problems in Diophantine approximations, Laurent and Nogueira \cite{LaurentNogueira12a} defined two exponents $\mu_\G(u,v)$ and $\hat\mu_\G(u,v)$ as follows. 
 \begin{Def}
The critical exponent $\mu_\G(u,v)$ is defined as the supremum of all $\alpha>0$ such that the set 
 $$\{\g\in \G: \|\g u-v\|<\|\g\|^{-\alpha}\}$$
 is unbounded. The uniform critical exponent, $\hat\mu_\G(u,v)$, is defined as the supremum over all $\alpha>0$ such that $ \G_T u\cap B_{1/T^{\alpha}}(v)\neq \emptyset$ for all sufficiently large $T$.
 Here $B_\delta(v)=\{u\in \R: \|u-v\|\leq \delta\}$ denote small norm ball with respect to some fixed norm on $\R^2$.
  \end{Def}
 Notice that $\hat{\mu}(u,v)\leq \mu(u,v)$ unless $v\in \G u$ (in which case $\hat\mu(u,v)=\infty$). Also, 
 as noted in \cite{LaurentNogueira12a,LaurentNogueira12b},  these exponents are invariant under the $\G\times\G$ action on $\R^2\times \R^2$, and by ergodicity they are constant almost everywhere.  We denote these constants by $\mu_\G$ and $\hat\mu_\G$ respectively.  
 
 In \cite{LaurentNogueira12a}, Laurent and Nogueira studied these exponents for $\G=\SL_2(\Z)$ and gave very precise estimates depending on the Diophantine properties of the slopes of $u$ and $v$. In particular, their analysis implies that for almost all $u,v\in \R^2\setminus\{0\}$ one has that $\frac{1}{3}\leq \hat\mu_{\SL_2(\Z)}(u,v)\leq \mu_{\SL_2(\Z)}(u,v)\leq \frac{1}{2}$
 and that $\mu_{\SL_2(\Z)}(u,v)\geq\frac{1}{3}$ holds for every target $v$ and any $u$ with a dense orbit.
 In particular, this implies that 
  \begin{equation}\label{e:SL2bounds}\frac{1}{3}\leq \hat\mu_{\SL_2(\Z)}\leq \mu_{\SL_2(\Z)}\leq \frac{1}{2}.\end{equation}
 Moreover, in \cite{LaurentNogueira12b} they showed that for any lattice $\G$, the upper bound $\mu_{\G}(u,v)\leq \frac{1}{2}$ holds for any $u$ with a dense orbit and a.e. $v\in \R^2$, so that  $\mu_{\G}\leq \frac{1}{2}$ for any lattice.
 
Another approach for this problem was given in \cite{MaucourantWeiss12}, where Maucourant and Weiss gave an effective version of the equidistribution result of $\G$-orbits, building on effective equidistribution of unipotent flows on $\G\bk \SL_2(\R)$. In particular, their results imply the following lower bound for the critical exponents of a generic orbit: For any lattice $\G$ in $\SL_2(\R)$,  for almost every $u\in \R^2$ (respectively, for all $u\in \R^2$ if $\G$ is cocompact),
$\hat\mu_\G(u,v)\geq \frac{1-2\tau}{144},$ for all $v\in \R^2\setminus\{0\}$. Moreover, for every $u\in \R^2$ with a dense orbit $\mu_\G(u,v)\geq \frac{1-2\tau}{144}$. Here $\tau=\tau(\G)\in[0,1/2]$ measures the spectral gap for $\G$, and in particular, $\tau(\G)=0$ for $\G=\SL_2(\Z)$ (see section \ref{s:gap} below for more details on the spectral gap). 

Recently Ghosh, Gorodnik, and Nevo \cite{GhoshGorodnikNevo14,GhoshGorodnikNevo15} studied a similar problem, in a more general setting, regarding rates of approximation of $\G$-orbits on homogenous spaces $X=G/H$ with $\G$  a lattice in a semisimple group $G$ and $H$ a closed subgroup. Their approach again builds on effective equidistribution results for the $H$ action on $\G\bk G$, but using the mean ergodic theorem instead of a pointwise ergodic theorem. A striking feature of their result is that it provides in many cases optimal rates of approximations.
In this note we borrow some of their ideas, as well as ideas of \cite{KelmerShrinking16} for proving an effective mean ergodic theorem for actions of unipotent groups, and \cite{GhoshKelmer15} relating mean ergodic theorems to shrinking target problems, to give bounds for the critical exponents of a generic orbit. Our main result is as follows:
 \begin{thm}\label{t:main}
 Let $\G\subseteq\SL_2(\R)$ be a lattice with spectral gap $\tau=\tau(\Gamma)$. Then
 \begin{enumerate}
 \item For any $v\in \R^2\setminus \{0\}$ for almost all $u\in \R^2$ 
 $$\frac{1-2\tau}{3}\leq \hat{\mu}_\G(u,v)\leq \mu_\G(u,v)\leq \frac{1}{2}.$$
 \item For almost all $u\in \R^2$ for any $v\in \R^2\setminus \{0\}$ we have  $\hat{\mu}_\G(u,v)\geq \frac{1-2\tau}{5}$.
 \end{enumerate}
  \end{thm}
  \begin{rem}
 When the representation of $G$ on $L^2_0(\G\bk G)$  is tempered (in particular for $\G=\SL_2(\Z)$) we have that $\tau(\G)=0$ and the first part implies that $\frac{1}{3}\leq \hat{\mu}_\G\leq \mu_\G\leq \frac{1}{2}$ recovering \eqref{e:SL2bounds}.  This is slightly better than the bound $\hat{\mu}_\G\geq \frac{1}{6}$  claimed in \cite{GhoshGorodnikNevo15} to be obtained by similar methods. For $\G$ a congruence lattices, using the best known bounds on the spectral gap, we get that $\frac{25}{96}\leq \hat\mu_\G\leq \frac{1}{2}$.  It is not unlikely that in fact $\hat\mu_\G=\mu_\G=\tfrac{1}{2}$ (independent of the spectral gap), however, proving this seems beyond our abilities at the moment. 
 \end{rem}

\begin{rem} 
We point out a subtle difference between the first part of our result, which holds for any target point but only for generic orbits, vs. the results of  \cite{LaurentNogueira12a}, that hold for any dense orbit, but the exponent depends on the slopes of the target point and the orbit. In particular, for $\G=\SL_2(\Z)$, if the target point $v\in \R^2$ has an irrational slope which is a Liouville number, the results of \cite[Theorem 2 (iii)]{LaurentNogueira12a} imply that $\hat\mu_{\SL_2(\Z)}(u,v)\geq \frac{1}{4}$ for almost all $u$, while we get $\hat\mu_{\SL_2(\Z)}(u,v)\geq \frac{1}{3}$. On the other hand, if the slope of $v$ is rational then \cite[Theorem 2 (ii)]{LaurentNogueira12a} imply that  $\hat\mu_{\SL_2(\Z)}(u,v)\geq \frac{1}{2}$ for almost all $u$, which is best possible.
\end{rem}

\begin{rem}
In the second part of our result, the bound for the critical exponent is weaker because we require that the orbit of a single point $u$ will approximate every target point simultaneously. Here the analysis of \cite{LaurentNogueira12a} imply that almost all $u\in \R^2\setminus\{0\}$ satisfy $\hat\mu_{\SL_2(\Z)}(u,v)\geq \frac{1}{4}$ for all $v\in \R^2\setminus\{0\}$, which is slightly better. However, our result holds for any lattice, and moreover, the method of proof generalizes to deal with the general problem of lattice action on homogenous spaces, thus answering the question of uniformity on a co-null set of orbits raised in \cite{GhoshGorodnikNevo14}.
\end{rem}

\begin{rem}
One can also consider the same problem for the action of lattices $\G\subseteq \SL_2(\C)$ acting on $\C^2$. There have been a few results in this case: for $\G=\SL_2(\cO)$ with $\cO=\Z[i]$ the ring of Gaussian integers recent results of Singhal  \cite{Singhal15} imply that $\frac{1}{3}\leq \hat\mu_{\SL_2(\cO)}\leq \mu_{\SL_2(\cO)}\leq \frac{1}{2}$, more generally, the work of  for Pollicott \cite{Pollicott11} give a lower bound for $\hat\mu_\G$ for any co-compact $\G$ in $\SL_2(\C)$. The methods of this paper could also be generalized to handle this case as well to show that $c_\G\leq \hat\mu_{\SL_2(\G)}\leq \mu_{\SL_2(\G)}\leq \frac{1}{2}$ for some explicit value of $c_\G$ depending on the spectral gap for $\G$.
\end{rem}


\section{Preliminaries and notation}
\subsection{Notation}
We write $A\ll B$ or $A=O(B)$ to indicate that $A\leq cB$
for some constant $c$.  If we wish to emphasize that constant depends on
some parameters we use subscripts, for example
$A\ll_\epsilon B$. We also write $A\asymp B$ to indicate that
$A\ll B\ll A$.

\subsection{Coordinates}
Let $G=\SL_2(\R)$ and consider the Cartan decomposition $G=NAK$ with $N$ unipotent, $A$ diagonal, and $K$ compact.
We will use the following coordinates
$$a_y=\left(\begin{smallmatrix} \sqrt{y} &0 \\ 0& 1/\sqrt{y}\end{smallmatrix}\right)\in A,\; 
n_x=\left(\begin{smallmatrix} 1& x\\ 0 &1\end{smallmatrix}\right)\in N,\; 
k_\theta=\left(\begin{smallmatrix} \cos(\theta) &\sin(\theta)\\ -\sin(\theta) & \cos(\theta)\end{smallmatrix}\right)\in K$$ 
In the coordinates $g=n_xa_yk_\theta$ the Haar measure of $G$ is $dg=\frac{dxdyd\theta}{y^2}$.

Let $\bar{n}_x=\left(\begin{smallmatrix} 1& 0\\ x &1\end{smallmatrix}\right)$ and let $\bar N=\{\bar n_x:x\in \R\}$. For any $g\in G$ apart from a set of measure zero we can also write $g=n_xa\bar n_{x'}$ and the Haar measure in these coordinates is given by $dg=\frac{dxdydx'}{y^2}$.

\subsection{Norms}
Fix a basis $\sB=\{X_{1},X_{2},X_{3}\}$ for the Lie algebra $\mathfrak{g}$ of $G$. Given a smooth test function $\psi\in C^{\infty}(\G\bk G)$,  define the ``$L^{p}$, order-$d$'' Sobolev norm $\cS_{p,d}(\psi)$  as
$$
\cS_{p,d}(\psi) \ : = \ \sum_{\ord(\sD)\le d}\|\sD\psi\|_{L^{p}(\G\bk G)}
.
$$
Here $\sD$ ranges over monomials in $\sB$ of order at most $d$ and $\sD$ acts on $\psi$ by left differentiation 
(e.g., $X\psi(g)=\frac{d}{dt}(\psi(ge^{tX}))|_{t=0}$). This definition depends on the basis, however, changing the basis $\sB$ only distorts
$\cS_{p,d}$ by a bounded factor.

\subsection{Spectral gap}\label{s:gap}
The group $G$ acts on the upper half plane $\bH=\{x+iy: y>0\}$ by linear fractional transformation preserving the hyperbolic metric. The (self adjoint extension of the) hyperbolic Laplacian $\triangle=-y^2(\frac{\partial^2}{\partial x^2}+\frac{\partial^2}{\partial x^2})$ acts on $L^2(\G\bk \bH)$, and its spectrum consists of a discrete part $0< \lambda_1\leq\lambda_2<\ldots$ and a continues part contained in $[\tfrac{1}{4},\infty)$ and spanned by Eisenstein Series (when $\G$ is non uniform). We say that $\G$ has a spectral gap $\tau=\tau(\G)\in[0,1/2]$, if $\lambda_1\geq \frac{1}{4}-\tau^2$. 

When $\G$ is a congruence group, Selberg's eigenvalue conjecture states that $\tau(\Gamma)=0$. This is known for  $\G=\SL_2(\Z)$ (as well as some other congruence groups of small level). The best known bound for a general congruence lattice is $\tau(\Gamma)\leq \frac{6}{64}$ \cite{KimSarnak2003}. On the other hand, if $\G$ is not a congruence lattice it is possible to have $\tau(\G)$ arbitrarily close to $1/2$.

\subsection{Decay of matrix coefficients}
Given a lattice $\G\subseteq G$ let $\mu$ denote the $G$ invariant probability measure on $\G\bk G$.
The group $G$ acts on the right on the space $L^2(\G\bk G,\mu)$ via $\pi(g)\psi(x)=\psi(x g)$, and for any two functions 
$\psi,\vf$ the corresponding matrix coefficient is 
$$\langle \pi(g)\psi,\vf\rangle=\int_{\G\bk G} \psi(xg)\vf(x)d\mu(x).$$

For $\psi,\vf\in L^2_0(\G\bk G)$  (the space orthogonal to the constant function) the corresponding matrix coefficients go to zero as $g\to\infty$, and the rate of decay is related to the spectral gap of $\G$ as follows (see \cite[Section 9.11]{Venkatesh2010}):
For any smooth $\psi,\vf\in L^2_0(\G\bk G)\cap C^\infty(\G\bk G)$ 
\begin{equation}\label{e:decay}
|\langle \pi(ka_y k')\psi,\vf\rangle|\ll_\epsilon (1+y)^{\tau-1/2+\epsilon}
(\cS_{2,1}(\psi)\cS_{2,1}(\vf)^{1/2+\epsilon}(\cS_{2,0}(\psi)\cS_{2,0}(\vf))^{1/2-\epsilon}
\end{equation}
where $\tau=\tau(\G)$ measures the spectral gap for $\G$.

\section{Proof of main results}
For the proof, we first use the duality of the $\G$ action on $G/\bar{N}\cong \R^2\setminus\{0\}$ and the $\bar{N}$ action on $\G\bk G$ to reduce the problem to a shrinking target problem for a unipotent flow. Then we prove an effective mean ergodic theorem and use it to give a partial solution for the shrinking target problem. Combining these results will give the proof of Theorem \ref{t:main}
\subsection{Reduction to a shrinking target problem}
To define our shrinking targets, fix $v=\left(\begin{smallmatrix} v_1\\ v_2\end{smallmatrix}\right)\in \R^2$ and assume that $v_1v_2\neq 0$. For small $\delta\in (0,1/2)$  consider the set $\cA_{\delta}=\cA_{\delta}(v)\subseteq G$ given by 
\begin{equation}\cA_{\delta}=\{n_{x}a_y\bar{n}_{x'}: |x'|<1/2,\; |\tfrac{1}{\sqrt{y}v_2}-1|\leq \tfrac{\delta}{2|v_2|},\; |\tfrac{x}{\sqrt{y} v_1}-1|\leq \tfrac{\delta}{2|v_1|}\}.
\end{equation}
Note that for any $g\in \cA_{\delta}(v)$ we have that  $\|g\left(\begin{smallmatrix} 0\\ 1\end{smallmatrix}\right)-v\|\leq \delta$ and moreover, if $\|g\left(\begin{smallmatrix} 0\\ 1\end{smallmatrix}\right)-v\|\leq \delta$ then $g \bar{n}_k\in \cA_\delta(v)$ for some $k\in \Z$. For each $\cA_\delta\subseteq G$ we define the corresponding set $\cB_{\delta}=\cB_\delta(v)\subseteq \G\bk G$ by
$$\cB_{\delta}=\{\G g: g\in \cA_{\delta}\}.$$
The following lemma shows that these shrinking targets $\cB_\delta(v)$ are stable under small perturbation in $v$.
\begin{lem}
If $v,\tilde v\in \R^2$ are off the axes and satisfy that $\|v-\tilde v\|<\delta$ then $\cB_{\delta}(v)\subseteq \cB_{2\delta}(\tilde v)$.
\end{lem}
\begin{proof}
Let $\G g\in \cB_\delta(v)$, then there is some $\g\in \G$ with $\g g\in \cA_\delta(v)$ and hence $\|\g g\left(\begin{smallmatrix} 0\\ 1\end{smallmatrix}\right)-v\|\leq \delta$. But then $\|\g g\left(\begin{smallmatrix} 0\\ 1\end{smallmatrix}\right)-v'\|\leq 2\delta$ and hence $\g g\bar{n}_k\in \cA_{2\delta}(\tilde v)$ for some $k\in \Z$. Now, on one hand $\g g=n_xa_y \bar{n}_{x'}$ with $|x'|\leq 1/2$ and on the other hand $\g g=n_{\tilde x}a_{\tilde y} \bar{n}_{\tilde x'-k}$ with $|\tilde{x}'-k|<1/2$ implying that $x=\tilde x, y=\tilde y, x'=\tilde{x}'$ and $k=0$, and hence $\g g\in \cA_{2\delta}(\tilde v)$ and $\G g\in \cB_{2\delta}(\tilde v)$ as claimed.
\end{proof}

The shrinking target problem is then to determine how fast can targets $\cB_{\delta_k}$ shrink so that the finite orbits 
$$\cO_k(x)=\{x\bar{n}_l: |l|\leq k\},$$ 
keeps hitting them.  The following lemma connects this shrinking target problem to the critical exponents (cf. \cite[Proposition 3.2]{GhoshGorodnikNevo15}).
 \begin{lem}\label{l:GammaToH}
Fix $g\in G$ and let $u=g\left(\begin{smallmatrix} 0\\ 1\end{smallmatrix}\right)\in \R^2$ and $x=\G g\in \G\bk G$.  For any $\alpha<\eta$
 \begin{enumerate}
\item If $\{\g\in \G:  \|\g u-v\|\leq \|\g\|^{-\eta}\}$ is unbounded then
  $\{k: x\bar{n}_k\in \cB_{1/k^{\alpha}}\}$ is unbounded.
  \item If $\cO_T(x)\cap \cB_{1/T^\eta}(v)\neq \emptyset$ for all sufficiently large $T$, then $\G_T u\cap B_{1/T^\alpha}(v)\neq \emptyset$
 for all sufficiently large $T$. \end{enumerate}
\end{lem}
\begin{proof}
Assume that $\g_i\in \G$ has $\|\g_i\|\to \infty$ and satisfies $ \|\g_i u-v\|\leq \|\g_i\|^{-\eta}$. Let $\delta_i=\|\g_i\|^{-\eta}$, then for each $i\in \N$ there is $k_i\in \Z$ such that $\g_i g \bar{n}_{k_i}\in \cA_{\delta_i}$ and hence $x\bar{n}_{k_i}\in \cB_{\delta_i}$. Moreover, since $\g_i g \bar{n}_{k_i}\in \cA_{\delta_i}$ and $\cA_{\delta_i}$ is contained in a compact set depending only on $v$, comparing norms we see that $\|\g_i\|\asymp_{g,v} \|\bar{n}_{k_i}\|\asymp k_i$. So there is a constant $c>0$ (depending on $g$ and $v$) such that 
$x\bar{n}_{k_i}\in \cB_{c/ k_i^{\eta}}$. Now, for any $\alpha<\eta$ we have that 
$\frac{c}{ k_i^{\eta}}\leq \frac{1}{ k_i^{\alpha}}$ for $k_i$ sufficiently large and so from some point  $x\bar{n}_{k_i}\in \cB_{1/ k_i^{\alpha}}$ and indeed the set $\{k: x\bar{n}_k\in \cB_{1/k^{\alpha}}\}$ is unbounded.

For the second statement, assume that for all $T\geq T_0$ there is $|k|\leq  T$ with $x\bar{n}_k\in \cB_{T^{-\eta}}$. Then there is $\g_k\in \G$ with $\g_k g\bar{n}_k\in \cA_{T^{-\eta}}$,  hence, $\|\g_k u-v\|\leq T^{-\eta}$. Also, as before, since $\g_k g\bar{n}_k\in \cA_{T^{-\eta}}$ comparing norms we get that $\|\g_k\|\asymp \|\bar{n}_k\|\asymp k$ so there is $c>0$ (depending on $g,v$) such that $\|\g_k\|\leq cT$. Setting $\tilde{T}=cT$, and $\tilde{T}_0=cT_0$, assuming that $\tilde{T}_0$ is sufficiently large so that $(\tilde{T}_0/c)^{-\eta}\leq \tilde{T}_0^{-\alpha}$, 
 we get that for all $\tilde{T}\geq\tilde{T}_0$ there is $\g\in \G_{\tilde{T}}$ with $\|\g u-v\|\leq (\tilde{T}/c)^{-\eta}\leq \tilde{T}^{-\alpha}$.
\end{proof}

\subsection{Solution of the shrinking target problem}
In this section we prove the following result, giving a partial solution to the shrinking target problem.
\begin{thm}\label{t:shrinking}
Fix $v\in \R^2$ with $v_1v_2\neq 0$ and let $\cB_\delta=\cB_\delta(v)$ be as above. Then
\begin{enumerate}
\item  If $\eta>1/2$ then for almost all $x\in \G\bk G$ the set $\{k\in \Z: x\bar{n}_k\in \cB_{k^{-\eta}}(v)\}$ is unbounded. 
\item If  $0<\eta<\frac{1-2\tau}{3}$, then for almost all $x\in \G\bk G$ there is $T_0>0$ such that for all $k\geq T_0$  we have that $\cO_k(x)\cap \cB_{k^{-\eta}}(v)\neq \emptyset$.
\item If  $0<\eta<\frac{1-2\tau}{5}$, then for any compact set $\Omega\subseteq\{v\in \R^2:v_1v_2\neq 0\}$, for almost all $x\in \G\bk G$ there is $T_0>0$ (depending on $x$ and $\Omega$) such that for all $k\geq T_0$  we have that $\cO_k(x)\cap \cB_{k^{-\eta}}(v)\neq \emptyset$ for all $v\in \Omega$.
\end{enumerate}
\end{thm}
\begin{rem}
This is a partial result because, even in the optimal setting when $\tau=0$, the lower bound $\eta>1/2$ in (1) is much larger than the upper bound $\eta<\frac{1-2\tau}{3}$ in (2). It is reasonable that the correct upper bound is also $\eta<1/2$ but we are not able to show this here. We note that for similar shrinking target problems, when the shrinking targets are spherical (i.e, right-$K$ invariant), by a similar argument one can get a lower bound that is the same as the upper bound. In fact this is shown for unipotent flows on more general homogenous spaces \cite{KelmerShrinking16}.  We also note that the exponent in (3) is even smaller because we require a much stronger form of approximation, that is, that a single orbit $\cO_k(x)$ approximate simultaneously all target points in $\Omega$. \end{rem}

Our main tool for the proof will be  an effective mean ergodic theorem for the unipotent flow $u_t=\bar{n}_t$ on $\G\bk G$ (we use the notation $u_t$ to indicate that the same results holds for any unipotent flow). For any $T>0$ let $\beta_T$ denote the averaging operator on $C^\infty_c(\G\bk G)$ given by
\begin{equation}\label{e:betaN}
\beta_T(\vf)(x)=\frac{1}{2T+1}\sum_{|k|\leq T}\vf(x u_k).
\end{equation}
Since the unipotent flow is ergodic, the mean ergodic theorem implies that $\|\beta_T(\vf)-\int_{\G\bk G}\vf d\mu)\|\to0$ as $T\to\infty$ for any  $\vf\in L^2(\G\bk G)$. Using the decay of matrix coefficients we show the following effective result.  
\begin{prop}\label{p:Ergodic}
Let $\tau=\tau(\G)$ measure the spectral gap for $\G$. Then 
for any smooth $\vf\in C^\infty_c(\G\bk G)$ we have 
$$\|\beta_T(\vf)-\int_{\G\bk G}\vf d\mu\|^2\ll_\epsilon \frac{\cS_{2,1}^{1+\epsilon}(\vf)\cS_{2,0}^{1-\epsilon}(\vf)}{T^{1-2\tau+\epsilon}}.$$
\end{prop}
\begin{proof}
Let $\vf_0=\vf-\int_{\G\bk G}\vf d\mu$ then $\vf_0\in L^2_0(\G\bk G)$ and $\beta_T(\vf)-\int_{\G\bk G}\vf d\mu=\beta_T(\vf_0)$. Now expand
\begin{eqnarray*}
\|\beta_T(\vf_0)\|^2&=&\langle \beta_T\vf_0,\beta_n\vf_0\rangle\\
&=&\frac{1}{(1+2T)^2}\sum_{|l|\leq T}\sum_{|k|\leq T}\langle \pi(u_k)\vf_0,\pi(u_l)\vf_0\rangle\\
&=&\frac{1}{(1+2T)^2}\sum_{|l|\leq T}\sum_{|k|\leq T}\langle \pi(u_{k-l})\vf_0,\vf_0\rangle\\
\end{eqnarray*}
Making a change of index summation and changing the order of summation we get
\begin{eqnarray*}
\|\beta_T(\vf_0)\|^2
&=&\frac{1}{(1+2T)^2}\sum_{|l|\leq T}\sum_{k=-T-l}^{T-l}\langle \pi(u_{k})\vf_0,\vf_0\rangle\\
&=&\frac{1}{(1+2T)^2}\sum_{|k|\leq 2T}\langle \pi(u_{k})\vf_0,\vf_0\rangle\#\{|l|\leq T: |l+k|\leq T\}\\
&\leq& \frac{1}{2T+1}\sum_{|k|\leq 2T}|\langle \pi(u_{k})\vf_0,\vf_0\rangle |
\end{eqnarray*}
Writing 
$u_t=ka_yk'$ with $y\geq 1$ and $k,k'\in K$ and comparing Hilbert-Schmidt norms we see that $2+t^2=y+y^{-1}$.  Using the decay of matrix coefficients \eqref{e:decay} we can bound
$$|\langle \pi(u_{k})\vf_0,\vf_0\rangle|\ll _\epsilon |k|^{2\tau-1+2\epsilon}
\cS_{2,1}^{1+2\epsilon}(\vf_0)\cS_{2,0}^{1-2\epsilon}(\vf_0),
$$
and hence
\begin{eqnarray*}
\|\vf_0\|^2&\ll_\epsilon & \frac{\cS_{2,1}^{1+2\epsilon}(\vf_0)\cS_{2,0}^{1-2\epsilon}(\vf_0)
}{2T+1}(1+2\sum_{k=1}^{2T}\frac{1}{|k|^{1-2\tau-2\epsilon}})\ll\frac{\cS_{2,1}^{1+2\epsilon}(\vf_0)\cS_{2,0}^{1-2\epsilon}(\vf_0)
}{T^{1-2\tau-2\epsilon}}.
\end{eqnarray*}
Finally, from orthogonality $\cS_{2,0}(\vf_0)\leq \cS_{2,0}(\vf)$ and since for any derivative $\cD\vf_0=\cD\vf$ we also have $\cS_{2,1}(\vf_0)\leq \cS_{2,1}(\vf)$, concluding the proof.
\end{proof}


Using the effective mean ergodic theorem as a variance estimate, we can estimate the measure of points whose orbit miss a small set $\cB_\delta$.
Explicitly, we show
\begin{prop}\label{p:finite}
Let $\cC_{T,\delta}=\{x\in \G\bk G: \cO_T(x)\cap \cB_\delta=\emptyset\}$. Then
$$\mu(\cC_{T,\delta})\ll_\epsilon \frac{1}{T^{1-2\tau+\epsilon}\delta^{3+\epsilon}}$$
\end{prop}
\begin{proof}
Let $\rho\in C^\infty_c(\R)$ be positive  supported in $(-1/2,1/2)$ with mean one. 
Define a function $f_\delta\in C^\infty_c(G)$ by 
$$f_\delta(n_x a_y \bar n_{x'})=\rho(\frac{x-\sqrt{y}{v_1}}{\delta})\rho(\frac{y-1/v_1}{\delta})\rho(x').$$
and let $F_\delta\in C^\infty_c(\G\bk G)$ be the corresponding $\G$-invariant function,
$$F_\delta(\G g)=\sum_{\g\in \G} f_\delta(\g g).$$ 
Clearly $f_\delta$ is supported on $\cA_\delta$ and $F_\delta$ is supported on $\cB_\delta$. 

Moreover, since $\cA_\delta\subseteq \cA_{1/2}$ is contained in some fixed compact set, there is some $C>0$ (depending only on $v$) such that $\cA_\delta$ is contained in a union of $C$ fundamental domains for $\G\bk G$. Consequently, we also have that $\mu(F_\delta)\asymp_v \mu(\cA_\delta) \asymp\delta^2$, that
$$\cS_{2,0}(F_\delta)\asymp_v (\int_{\cF}|f_\delta|^2d\mu)^{1/2}\asymp \sqrt{\mu(\cA_\delta)}\asymp \delta$$
and similarly $\cS_{2,1}(F_\delta)\asymp_v 1$. With these estimates, Proposition  \ref{p:Ergodic} implies that 
$$\| \beta_T(F_\delta)-\mu(F_\delta)\|^2\ll_\epsilon \delta^{1-\epsilon} T^{2\tau-1-\epsilon}.$$
On the other hand, since $\beta_T(F_\delta)(x)=0$ for all $x\in \cC_{T,\delta}$ we can bound from bellow 
$$\| \beta_T(F_\delta)-\mu(F_\delta)\|^2\geq \int_{\cC_{T,\delta}}|\beta_T(F_\delta)-\mu(F_\delta)|^2d\mu=\mu(\cC_{T,\delta})\delta^4$$
from which the result follows.
\end{proof}

We now go back to the shrinking target problem and give the 
\begin{proof}[Proof of Theorem \ref{t:shrinking}]
First, noting that $\mu(\cB_{\delta})\asymp \delta^2$ if we take $\delta_k= k^{-\eta}$ with $\eta>1/2$ the series 
$$\sum_k\mu(\cB_{\delta_k})\asymp \sum_k \frac{1}{k^{2\eta}}<\infty,$$
converges and hence by the easy half of the Borrel-Cantelli lemma, for almost all $x\in \G\bk G$ we have that  $\{k: x\bar{n}_k\in \cB_{\delta_k}\}$ is bounded.

Next, for the lower bound for a fixed target point $v$, assume that $\delta_k=k^{-\eta}$ with $0<\eta<\frac{1-2\tau}{3}$ and let $\cC_v\subseteq \G\bk G$ denote the set of all points such that for any $T\in \N$ there is $k\geq T$ with $\cO_k(x)\cap B_{\delta_k}(v)=\emptyset$,
that is,
$$\cC_v=\bigcap_{T\in \N}\bigcup_{k\geq T}\cC_{k},$$
with $\cC_k=\cC_{k,\delta_k}$ the set of points for which $\cO_k(x)\cap \cB_{\delta_k}(v)=\emptyset$.
Now consider the sets
$$\tilde\cC_k=\{x\in \G\bk G: \cO_{k}(x)\cap  \cB_{\delta_{2k}}(v)=\emptyset\},$$
and note that, since the orbits $\cO_{k}(x)$ are increasing sets and the targets $\cB_{\delta_k}(v)$ are decreasing, we have that 
$$\bigcup_{l=k}^{2k}\cC_l=\{x\in \G\bk G: \exists  l\in[k, 2k],\; \cO_{l}(x)\cap \cB_{\delta_l}(v)=\emptyset\}\subseteq
\tilde\cC_k.$$
We thus have that 
$$\cC_v=\bigcap_{T\in \N}\bigcup_{l\geq \log(T)}\bigcup_{k=2^l}^{2^{l+1}}\cC_{k}\subseteq \bigcap_{T\in \N}\bigcup_{l\geq \log(T)}\tilde\cC_{2^l}.$$
By Proposition \ref{p:finite} we can estimate
$\mu(\tilde\cC_{2^l})\ll_{\epsilon} \frac{1}{2^{l(1-2\tau-3\eta+\epsilon)}}$
and hence
 \begin{eqnarray*}
 \mu(\bigcup_{l\geq \log(T)}\tilde\cC_{2^l})&\leq& \sum_{l\geq \log{T}}\mu(\tilde\cC_{2^l})\\
 &\ll_\epsilon&  \sum_{l\geq \log{T}}\frac{1}{2^{l(1-2\tau-3\eta+\epsilon)}}.
 \end{eqnarray*}
Now, from our assumption $1-2\tau-3\eta>0$ and taking $\epsilon=\frac{1-2\tau-3\eta}{2}$ we can estimate
\begin{eqnarray*}
 \mu(\bigcup_{l\geq \log(T)}\tilde\cC_{2^l})\ll  \sum_{l\geq \log(T)}2^{-\epsilon}\ll T^{-\epsilon}
\end{eqnarray*}
implying that $\mu(\cC_v)=0$.

Finally, for the uniform bound, let $\delta_k=k^{-\alpha}$ with  $0<\eta<\alpha<\frac{1-2\tau}{5}$ and for each $k$ let $\{v_{k,i}\}_{i=1}^{m_k}\subseteq \Omega$ be $\delta_k$-dense in $\Omega$ (so that $m_k\asymp_\Omega \delta_k^{-2}$). 
Moreover, we choose our points so that for $k=2^l$ the set $\{v_{k,i}\}_{i=1}^{m_k}$ contains all points $v_{k',i}$ with $k'\leq k$. Now,  let $\cC_\Omega\subseteq \G\bk G$ denote the set of all points such that for any $T\in \N$ there is $k\geq T$ with $\cO_k(x)\cap B_{\delta_k}(v_{k,i})=\emptyset$ for some $1\leq i\leq m_k$, that is, 
$$\cC_\Omega=\bigcap_{T\in \N}\bigcup_{k\geq T}\bigcup_{i=1}^{m_k}\cC_{k,i},$$
with $\cC_{k,i}=\{x\in \G\bk G: \cO_k(x)\cap B_{\delta_k}(v_{k,i})=\emptyset\}$. As before, we have 
$$\cC_\Omega\subseteq \bigcap_{T\in \N}\bigcup_{l\geq \log(T)}\bigcup_{i=1}^{m_{2^{l+1}}}\tilde\cC_{2^l,i},$$
with $\tilde\cC_{k,i}=\{x\in \G\bk G: \cO_{k}(x)\cap  \cB_{\delta_{2k}}(v_{2k,i})=\emptyset\}$. By Proposition \ref{p:finite}, for each $i=1,\ldots, m_{2^{l+1}}$ we can estimate,
$\mu(\tilde\cC_{2^l,i})\ll_{\epsilon} \frac{1}{2^{l(1-2\tau-3\alpha+\epsilon)}}$, and since for each $l$ there are $\asymp 2^{2\alpha l}$ such sets
we can bound 
 \begin{eqnarray*}
 \mu(\cC_\Omega) &\ll_\epsilon&  \sum_{l\geq \log{T}}\frac{1}{2^{l(1-2\tau-5\alpha+\epsilon)}}.
 \end{eqnarray*}
Taking $\epsilon=\frac{1-2\tau-5\alpha}{2}>0$ we get that $\mu(\cC_\Omega) \ll T^{-\epsilon}$ for all $T>0$ and hence $\mu(\cC_\Omega)=0$. Now let $x\in \G\bk G,\; x\not\in\cC_\Omega$ and let $v\in \Omega$.
For each $k$ let $v_{k,i}$ be $\delta_k$  close to $v$. Then for all sufficiently large $k$, $\cO_k(x)\cap \cB_{\delta_k}(v_{k,i})\neq \emptyset$. Since $\norm{v-v_{k,i}}\leq \delta_k,$ we have that, for $k$ sufficiently large, $\cB_{\delta_k}(v_{k,i})\subseteq \cB_{2\delta_k}(v)\subseteq \cB_{k^{-\eta}}(v)$, implying that $\cO_k(x)\cap \cB_{k^{-\eta}}(v)\neq \emptyset$ as well. 
\end{proof}

\subsection{Conclusion}

Now combining the shrinking target results in Theorem \ref{t:shrinking} with Lemma \ref{l:GammaToH} we get estimates on the critical exponents giving the
\begin{proof}[Proof of Theorem \ref{t:main}]
Let $0\neq v\in \R^2$ and (perhaps after replacing $v$ by $\g v$ for some $\g\in \G$) we may assume that $v_1v_2\neq 0$.
For any $\delta\in (0,1/2)$ let $\cA_\delta=\cA_\delta(v)$ and $\cB_\delta=\cB_\delta(v)$ be as above.

First to show that for almost all $u\in \R^2$ we have $\mu(u,v)\leq 1/2$ fix some $\eta>1/2$ and let $U\subseteq \R^2$ denote the set of all $u\in \R^2$ such that there is a sequence $\g_k$ with $\|\g_ku-v\|\leq \|\g_k\|^{-\eta}$. 
For each $u\in U$ let 
$g_u=\left(\begin{smallmatrix} u_2^{-1} & u_1\\ 0 & u_2\end{smallmatrix}\right)$ and let $\cU\subseteq G$ be defined by 
$\cU=\{g_u\bar{n}_x: u\in U,\; |x|\leq 1/2\}$. Let $\eta>\alpha>1/2$, then by the first part of Lemma \ref{l:GammaToH}, for any $g\in \cU$, the set $\{k: x\bar{n}_k\in \cB_{k^{-\alpha}}\}$ is unbounded, and hence 
$$\{\G g: g\in U\}\subseteq \{x\in \G\bk G: \{k: x\bar{n}_k\in \cB_{k^{-\alpha}}\} \mbox{ is unbounded}\}.$$
By the first part of Theorem \ref{t:shrinking} the set on the right has measure zero.  We thus get that the set $\{\G g:g\in \cU\}\subset \G\bk G$ is a null set. But then the set $\cU\subset G$ and hence also $U\subseteq \R^2$ must also have measure zero. 
This shows that for almost all $u\in \R^2$ the set $\{\g: \|\g u-v\|\leq \|\g\|^{-\eta}\}$ is bounded so $\mu(u,v)\leq \eta$ for almost all $u\in \R^2$. Since this holds for any $\eta>1/2$ we get the upper bound  $\mu(u,v)\leq 1/2$ for almost all $u\in \R^2$.

Next to show that for any $v\in \R^2$ for almost all $u\in \R^2$ we have $\hat\mu(u,v)\geq \frac{1-2\tau}{3}$ fix some $\alpha<\frac{1-2\tau}{3}$. 
Let $\alpha<\eta<\frac{1-2\tau}{3}$, let $\delta_k=k^{-\eta}$ and let $\cC\subset \G\bk G$ be the set of all points $x\in \G\bk G$ such that for every $T\in \N$ there is $k\geq T$ with $\cO_k(x)\cap B_{\delta_k}=\emptyset$. Then for any $g\in G\setminus \cC_0$ and any $u=g\left(\begin{smallmatrix} 0\\ 1\end{smallmatrix}\right)$ by the second part of Lemma \ref{l:GammaToH}, $\G_T u\cap B_{1/T^\alpha}(v)\neq \emptyset$ for all sufficiently large $T$.  By the second part of Theorem \ref{t:shrinking} we have that $\mu(\cC)=0$ and hence  the set $\cC_0=\{g\in G: \G g\in \cC\}$ is a null set and the set $\{g\left(\begin{smallmatrix} 0\\ 1\end{smallmatrix}\right): g\in G\setminus \cC_0\}$ is  set of full measure. This shows that for almost all $u\in \R^2$, we have that $\G_T u\cap B_{1/T^\alpha}(v)\neq \emptyset$ for all sufficiently large $T$, and hence $\hat\mu(u,v)\geq \alpha$. Since this holds for any $\alpha<\frac{1-2\tau}{3}$ we get that $\hat\mu(u,v)\geq\frac{1-2\tau}{3}$ for almost all $u\in \R^2$.

Finally, for the uniform bound, let $\Omega_\infty=\{v\in \R^2: |v_1|\geq 1,\; |v_2|\geq 1\}$. Recall that any orbit $\G v$ is either dense or a lattice, and hence must intersect $\Omega_\infty$, and since $\hat\mu(u,v)=\hat\mu( u,\g v)$ it is enough to consider target points $v\in \Omega_\infty$. Next, since we can write $\Omega_\infty=\bigcup\Omega_i$ as a union of countably many compact sets (all bounded away from the axis), it is enough to show for each $i\in\N$, for almost all $u\in \R^2$ we have that $\hat\mu(u,v)\geq \frac{1-2\tau}{5}$ for all $v\in \Omega_i$.  This follows from the third part of  Theorem \ref{t:shrinking} by the same argument as above.
\end{proof}

%


\begin{thebibliography}{GGN15}

\bibitem[GGN14]{GhoshGorodnikNevo14}
A.~{Ghosh}, A.~{Gorodnik}, and A.~{Nevo}.
\newblock {Best possible rates of distribution of dense lattice orbits in
  homogeneous spaces}.
\newblock J. Reine Angew. Math. to appear.

\bibitem[GGN15]{GhoshGorodnikNevo15}
Anish Ghosh, Alexander Gorodnik, and Amos Nevo.
\newblock Diophantine approximation exponents on homogeneous varieties.
\newblock In {\em Recent trends in ergodic theory and dynamical systems},
  volume 631 of {\em Contemp. Math.}, pages 181--200. Amer. Math. Soc.,
  Providence, RI, 2015.
  
 \bibitem[GK15]{GhoshKelmer15}
A.~{Ghosh} and D.~{Kelmer}, \emph{{Shrinking Targets for Semisimple Groups}},
  ArXiv e-prints (2015).

\bibitem[GW07]{GorodnikWeiss07}
Alex Gorodnik and Barak Weiss.
\newblock Distribution of lattice orbits on homogeneous varieties.
\newblock {\em Geom. Funct. Anal.}, 17(1):58--115, 2007.

\bibitem[Kel16]{KelmerShrinking16}
D.~Kelmer
\newblock Shrinking target problems for unipotent flows
\newblock in preparation.


\bibitem[KS03]{KimSarnak2003}
H.~Kim and P.~Sarnak.
\newblock Refined estimates towards the {R}amanujan and {S}elberg conjectures.
\newblock {\em J. Amer. Math. Soc.}, 16(1):175--181, 2003.

\bibitem[Led99]{Ledrappier99}
Fran{\c{c}}ois Ledrappier.
\newblock Distribution des orbites des r\'eseaux sur le plan r\'eel.
\newblock {\em C. R. Acad. Sci. Paris S\'er. I Math.}, 329(1):61--64, 1999.

\bibitem[LN12a]{LaurentNogueira12a}
Michel Laurent and Arnaldo Nogueira.
\newblock Approximation to points in the plane by {${\rm SL}(2,\Bbb
  Z)$}-orbits.
\newblock {\em J. Lond. Math. Soc. (2)}, 85(2):409--429, 2012.

\bibitem[LN12b]{LaurentNogueira12b}
Michel Laurent and Arnaldo Nogueira.
\newblock Inhomogeneous approximation with coprime integers and lattice orbits.
\newblock {\em Acta Arith.}, 154(4):413--427, 2012.

\bibitem[MW12]{MaucourantWeiss12}
Fran{\c{c}}ois Maucourant and Barak Weiss.
\newblock Lattice actions on the plane revisited.
\newblock {\em Geom. Dedicata}, 157:1--21, 2012.

\bibitem[Nog02]{Nogueira02}
Arnaldo Nogueira.
\newblock Orbit distribution on {$\Bbb R^2$} under the natural action of {${\rm
  SL}(2,\Bbb Z)$}.
\newblock {\em Indag. Math. (N.S.)}, 13(1):103--124, 2002.

\bibitem[Pol11]{Pollicott11}
M.~Pollicott, 
\newblock Rates of convergence for linear actions of cocompact lattices on the complex plane
  \newblock {\em Integers} \textbf{11B} (2011), Paper No. A12, 7.
  
 \bibitem[Sin15]{Singhal15}
L.~{Singhal},
\newblock Diophantine exponents for standard linear actions of
  $\mathrm{SL}\_2$ over discrete rings in $\mathbb{C}$
 \newblock ArXiv e-prints (2015).

\bibitem[Ven10]{Venkatesh2010}
Akshay Venkatesh.
\newblock Sparse equidistribution problems, period bounds and subconvexity.
\newblock {\em Ann. of Math. (2)}, 172(2):989--1094, 2010.

\end{thebibliography}
\end{document}